\documentclass[reqno, 10pt]{amsart}

\usepackage{amsmath,amssymb,amsthm, booktabs, tikz}
\usepackage{graphicx} 
\usepackage{algorithm}
\usepackage{algorithmicx}
\usepackage{algpseudocode}
\usepackage{url}

\newtheorem*{corollary}{Corollary}
\newtheorem*{thm}{Theorem}
\newtheorem{lemma}{Lemma}
\newtheorem*{fact}{Fact}

\title[]{Gradient descent with exponentially\\ increasing stepsizes and restarts}

\author[]{Fran\c{c}ois Cl\'ement}
\address{Department of Mathematics, University of Washington, Seattle}
\email{fclement@uw.edu}

\author[]{Stefan Steinerberger}
\address{Department of Mathematics and Department of Applied Mathematics, University of Washington, Seattle}
\email{steinerb@uw.edu}

\begin{document}
\begin{abstract} Let $f:\mathbb{R}^d \rightarrow \mathbb{R}$.
    We consider gradient descent $x_{n+1} = x_n - \tau_n \nabla f(x_n)$, where the stepsize $\tau_n = \tau \cdot e^{rn}$ is exponentially growing (with $\tau > 0$ and $0 < r \ll 1$). This diverges for almost all initial values.  We show that restarting the algorithm whenever $\|x_{n+1} - x_n\| \geq e^r\|x_n - x_{n-1}\|$ has good properties: it works very well in practice; we determine the limiting convergence rate in the case of convergence to a non-degenerate local minimum: it improves on classic gradient descent even though computational cost is comparable. The precise choice of $0 < r \ll 1$ does not matter much and the method is virtually independent of an initial stepsize $\tau$ that is too small:
    while the convergence rate for gradient descent decays linearly as $\tau \rightarrow 0$, it decays as $1/\log(1/\tau)$ in this modified version; numerical examples illustrate the results.
\end{abstract}

\maketitle

\vspace{-10pt}

\section{Introduction and Results}
\subsection{Introduction}
Let $f:\mathbb{R}^d \rightarrow \mathbb{R}$. We consider gradient descent
$$ x_{n+1} = x_n - \tau_n \nabla f(x_n)$$
with $\tau_n = e^{rn} \tau$ where $\tau > 0$ is an initial guess for the step-size and $0 < r \ll 1$ is a small positive parameter. One would not expect such a method to converge and it is easy to see that it does not. However, in what may appear tautological, it does \textit{very} well just before it starts to fail. Consider the one-dimensional example $f(x) = x^2/2$ with initial guess $x_0 = 1$ and stepsize $\tau =0.1$. Standard gradient descent leads to $x_n = 0.9^n$.  In comparison, using exponential stepsizes leads to the closed-form expression
$$ x_{n+1} = \prod_{k=0}^{n} \left(1 - \frac{e^{rk}}{10} \right).$$

\begin{center}
\begin{figure}[h!]
    \begin{tikzpicture}
        \node at (0,0) {\includegraphics[width=0.5\textwidth]{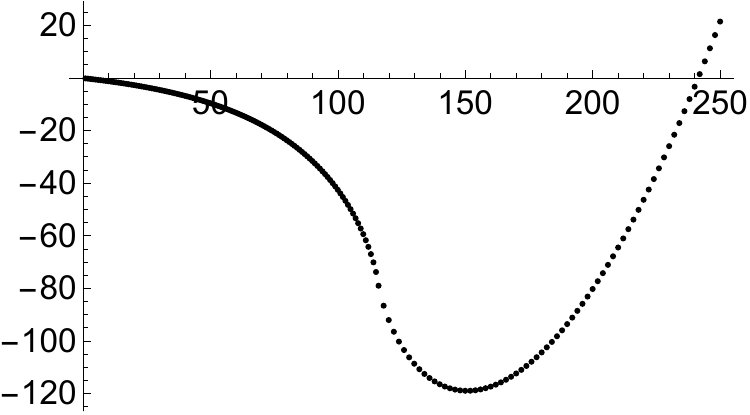}};
        \node at (5,0) {\Large $\log \left|\prod_{k=0}^{n} \left(1 - \frac{e^{0.02 k}}{10}\right) \right| $};
    \end{tikzpicture}
    \vspace{-10pt}
    \caption{The product in question for $r=0.02$.}
    \label{fig:1}
\end{figure}    
\end{center}

It is possible that $r>0$ may be chosen such that the product happens to be $0$ for all sufficiently large $n$, however, this only happens if $e^{r k} = 10$ for some $k \in \mathbb{N}$. For all other $r$ the product will eventually diverge.
Nonetheless, see Fig. 1, it is noteworthy that the product does get \textit{very} close to 0 before it converges.  While standard gradient descent reaches a value of $0.9^{250} \sim e^{-26} \sim 10^{-12}$ with the first 250 steps, this exponentially accelerated gradient descent reaches values of up to $e^{-120} \sim 10^{-52}$ within the first 250 steps around $n \sim 150$ (see Fig. 1). We were motivated by the question of whether this fact could somehow be exploited.

\subsection{Algorithm}
We propose a very simple algorithm: use  
$$ x_{n+1} = x_n -  e^{rn} \tau   \nabla f(x_n)$$
until the condition $\| x_{n+1} - x_n\| \leq e^r\|x_n - x_{n-1}\|$ is violated. One could alternatively think of stopping the algorithm once
$ \| \nabla f(x_n)\| \leq  \| \nabla f(x_{n-1})\|$ is violated.
Once the condition is triggered, we restart the algorithm with $x_0 = x_n$. A formal implementation of the algorithm is as follows.

\begin{algorithm}[h!]
\caption{Exponential Gradient Descent with Restarts }
\begin{algorithmic}
    \Require A function $f:\mathbb{R}^d \rightarrow \mathbb{R}$, an initial point $x_0 \in \mathbb{R}^d$, an initial stepsize $\tau > 0$, a multiplier $r>0$ and a desired number of steps $n \in \mathbb{N}$.
    \State Compute $x_1 = x_0 - \tau \cdot \nabla f(x_0)$.
    \State Set $k = 1$.
    \For{$2 \leq j \leq n$} 
        \State $y = x_{j-1} - \tau e^{r k} \cdot \nabla f(x_{j-1})$
        \If {$\|y-x_{j-1}\| \leq e^r\|x_{j-1}-x_{j-2}\|$ }
        \State $x_j = y$
       \State $k = k+1$
        \Else
        \State $x_j = x_{j-1} - \tau \nabla f(x_{j-1})$
        \State $k = 1$ 
        \EndIf
        \EndFor
  \end{algorithmic}
\end{algorithm}
This is computationally about as expensive as standard gradient descent: the new ingredient is the computation of the gap $\|x_{n} - x_{n-1}\|$, however, since that vector is computed anyway to update $x_{n-1}$ to $x_n$, the cost of computing its norm is negligible.  Like standard gradient descent, the algorithm requires a stepsize $r>0$ to be chosen, however, the algorithm exhibits a great deal of stability with respect to the choice of $r$ (see below for examples); this is fully explained by our analysis. If $r$ is too small, then the method behaves like standard gradient descent; if $r$ is too large, then many restarts will be triggered (which is not a priori a bad thing) and the method also behaves like standard gradient descent. In summary,
\begin{enumerate}
    \item we describe a method that is computationally about as expensive as standard gradient descent
    \item it has an additional parameter $0 < r \ll 1$ but its behavior is virtually independent of $r$ (in a way that can and will be made precise)
    \item it leads to significant speedups, especially when the condition number is large or the stepsize $\tau$ is small (see \S 1.4)
    \item and, empirically, it works well for convex and nonconvex problems (\S 4).
\end{enumerate}

\subsection{Main Result}
We can now state the main result which describes the asymptotic behavior of the algorithm close to a non-degenerate local minimum: the convergence rate is exponential. Moreover, in a fairly concrete sense, in practice that rate does \textit{not} actually depend on $r$ (in a way that will be made precise in \S 4 and \S 5). This means that, just as in the case of classical gradient descent, the exponential rate depends only on the largest and smallest eigenvalue of the quadratic form near the minimum as well as the stepsize $\tau$.

\begin{thm}[Main Result] Let $Q \in \mathbb{R}^{d \times d}$, $d \geq 2$, be a symmetric positive-definite matrix with eigenvalues $\lambda_{\max} \geq \dots \geq \lambda_{\min} > 0$, let $f(x) = \left\langle x, Qx \right\rangle$, let $x_0 \in \mathbb{R}^d$ and let $\tau < 1/\lambda_{\max}$. Then, for all $0 < r < r_0(Q, x_0, \tau)$ sufficiently small (up to an exceptional set), we have, for all $n$ sufficiently large,
$$ \frac{1}{n} \log{\|x_n\|} = - (1+o(1)) \cdot c(\lambda_{\max}, \lambda_{\min}, \tau) ,$$
where $o(1)$ is with respect to $r \rightarrow 0$. Moreover, the constant $c(\lambda_{\max}, \lambda_{\min}, \tau)$ has an explicit description: using $\Phi(y)$ to denote Spence's function
$$ \Phi(y) = - \int_0^y \frac{\log |1-z|}{z} dz$$
and $x \in \mathbb{R}_{>0}$ to be the solution of 
$$ \Phi(\tau \lambda_{\max}) - \Phi(\tau \lambda_{\max} e^{x}) = \Phi(\tau \lambda_{\min}) - \Phi(\tau \lambda_{\min} e^{x}),$$
we have 
$$ c(\lambda_{\max}, \lambda_{\min}, \tau) = \frac{1}{x}  \left( \Phi(\tau \lambda_{\min} e^{x}) - \Phi(\tau \lambda_{\min}) \right).$$
Moreover, the asymptotic density of restarts is given by $(1+o(1)) \cdot r/x$.
\end{thm}

\textbf{Comments.} 
\begin{enumerate}
    \item The result is stated for quadratic forms, however, in practice it very accurately predicts the behavior of the method near global minima (perhaps not surprising, there is a great deal of robustness and the higher order terms are vanishingly small); see below for an example.
    \item The `exceptional set' of values of $r$ is necessary; however, the set is benign in the sense that if
    $r$ is chosen to be in the set, then the convergence rate will be even faster (see \S 5.2 for an example). The set is extremely small; constructing an element would require knowledge of the eigenvalues of $Q$ which, in practice, are not known.
    \item The asymptotic behavior is determined as $r \rightarrow 0^+$, however, the convergence rate in $r$ is remarkably benign (see also \S 4 and \S 5 for the reason).
\end{enumerate}

We illustrate the result for a concrete toy problem: the Rosenbrock function \cite{rosenbrock}
$$f(x,y) = x^2 + 100 (y-x^2)^2.$$
It has a strict global minimum in $(0,0)$ but its `banana-shaped' level sets can pose a challenge for gradient descent algorithms.
 We consider the behavior under 50 initial points that are equispaced on the unit circle centered around the global minimum, i.e.
$ x_0 = \left( \cos\left( 2\pi k/50 \right), \sin\left( 2\pi k/50 \right) \right)$ for $k=1, \dots, 50$.
The natural point of comparison is standard gradient descent $x_{n+1} = x_n - \tau \nabla f(x_n)$. There is the issue of choosing the stepsize $\tau$. Testing numerically, we see divergence at $\tau = 0.003$ and pick $\tau =0.001$. The asymptotic convergence rate is then given by $(1-\tau \lambda_{\min}((D^2f)(0,0)))^n \sim e^{-0.002n}$ and this is what is observed in Fig. \ref{fig:2} (left).

\begin{center}
    \begin{figure}[h!]
       \begin{tikzpicture}
    \node at (0,0) {\includegraphics[width=0.42\linewidth]{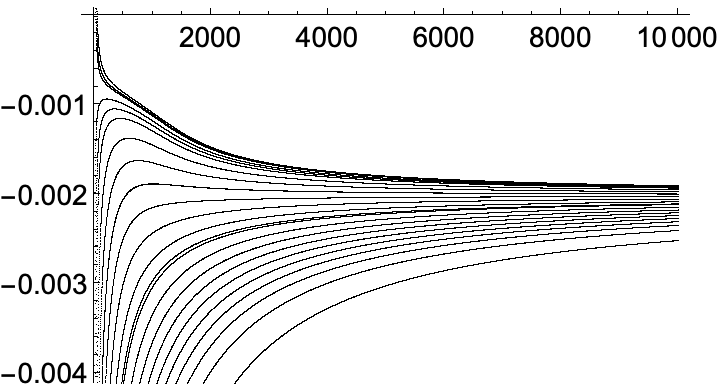}};
        \node at (6,0) {\includegraphics[width=0.42\linewidth]{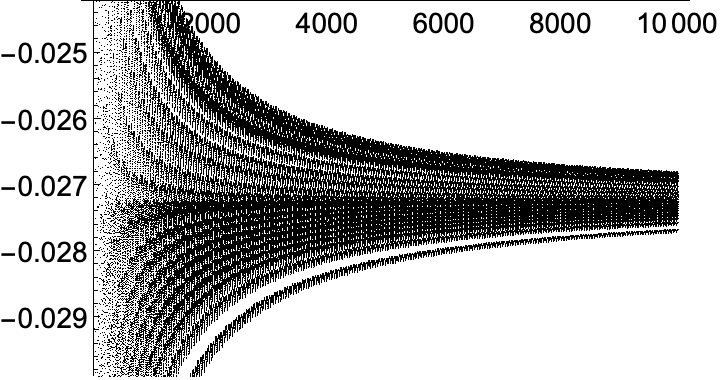}};
       \end{tikzpicture}
        \caption{The behavior of $n^{-1} \log \|x_n - x^*\|$ for the usual gradient descent (left) and Algorithm 1 (right) for 50 initial values.}
        \label{fig:2}
    \end{figure}
\end{center}

We compare this to the performance of Algorithm 1 with the same value $\tau = 0.001$. We set $r=0.1$ and see, numerically in Fig. \ref{fig:2}, a nearly ten-fold increase in convergence speed. Using that the eigenvalues at the critical point are $200$ and $2$, the Theorem asks us to solve the equation 
$$  \Phi(0.2 ) - \Phi(0.2 e^{x}) = \Phi(0.002) - \Phi(0.002 e^{x}),$$
which a standard numerical root finder finds to have the solution $x \sim 4.0072$. The Theorem now predicts a convergence rate of $\exp(- c n)$ with the constant given by
$$ c(200,2, 0.001) = \frac{1}{x} \left( \Phi(0.002) - \Phi(0.002 e^{x})\right) = 0.0277415\dots$$
which is what is observed in Fig. \ref{fig:2} (right). We had to pick a new parameter, $r$. However, the method is \textit{very} robust and we get similar rates for different values of $r$ (this is not mysterious and explained by the proof, see \S 3).

\begin{center}
    \begin{figure}[h!]
       \begin{tikzpicture}
    \node at (0,0) {\includegraphics[width=0.3\linewidth]{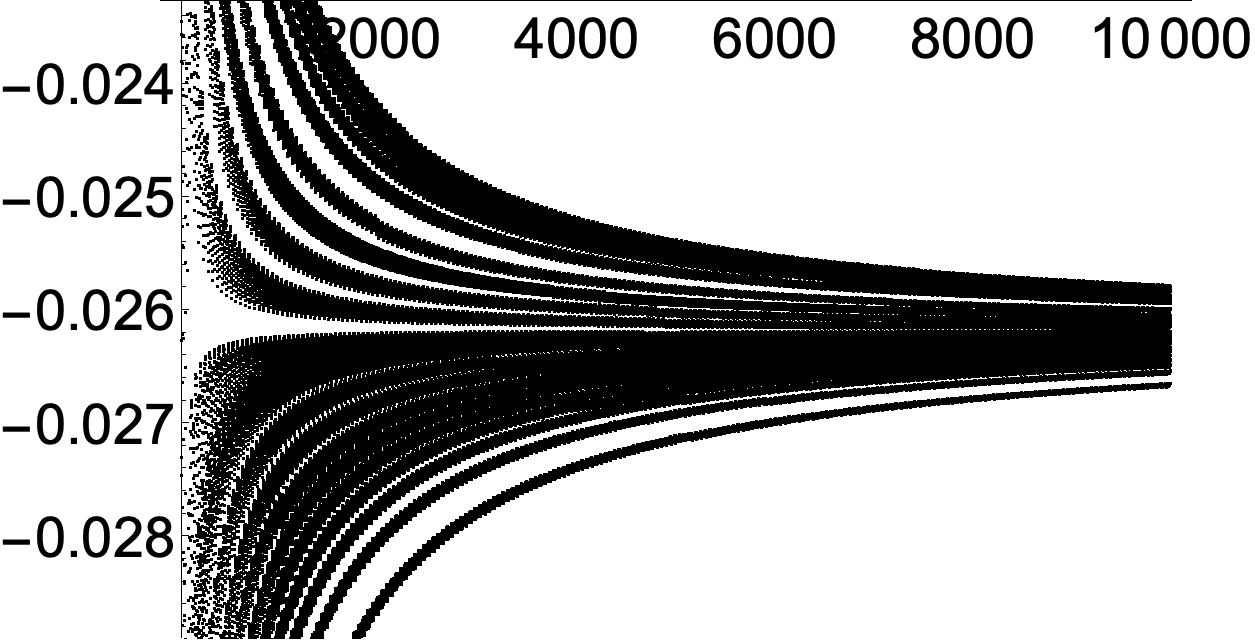}};
        \node at (4,0) {\includegraphics[width=0.3\linewidth]{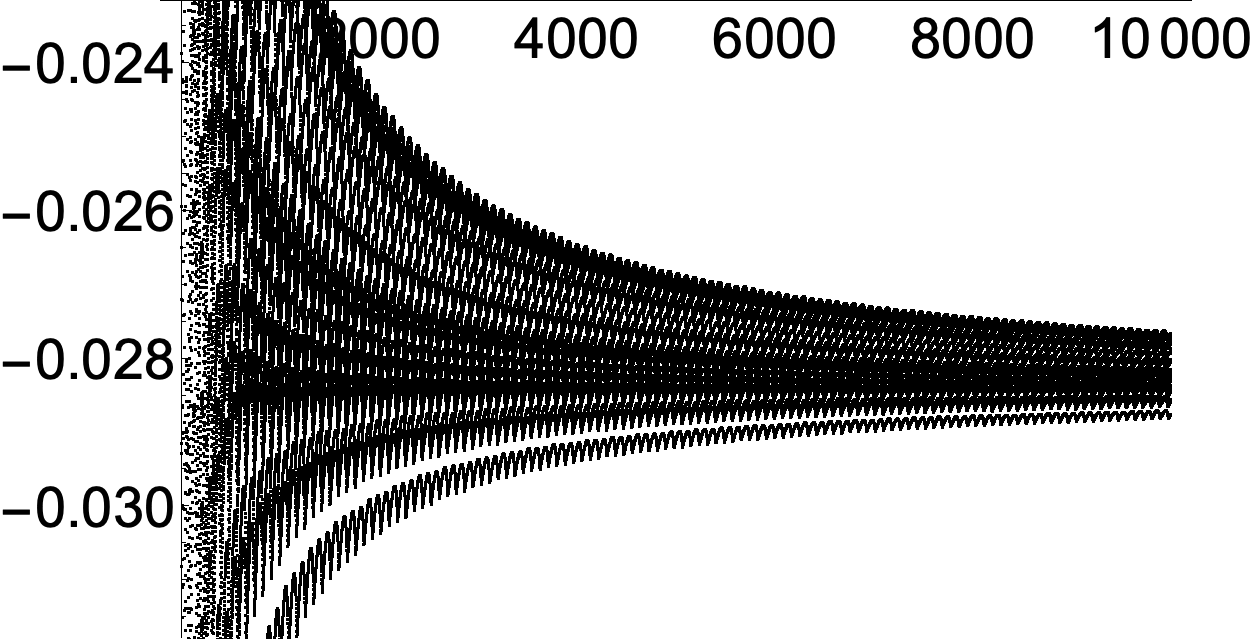}};
    \node at (8,0) {\includegraphics[width=0.3\linewidth]{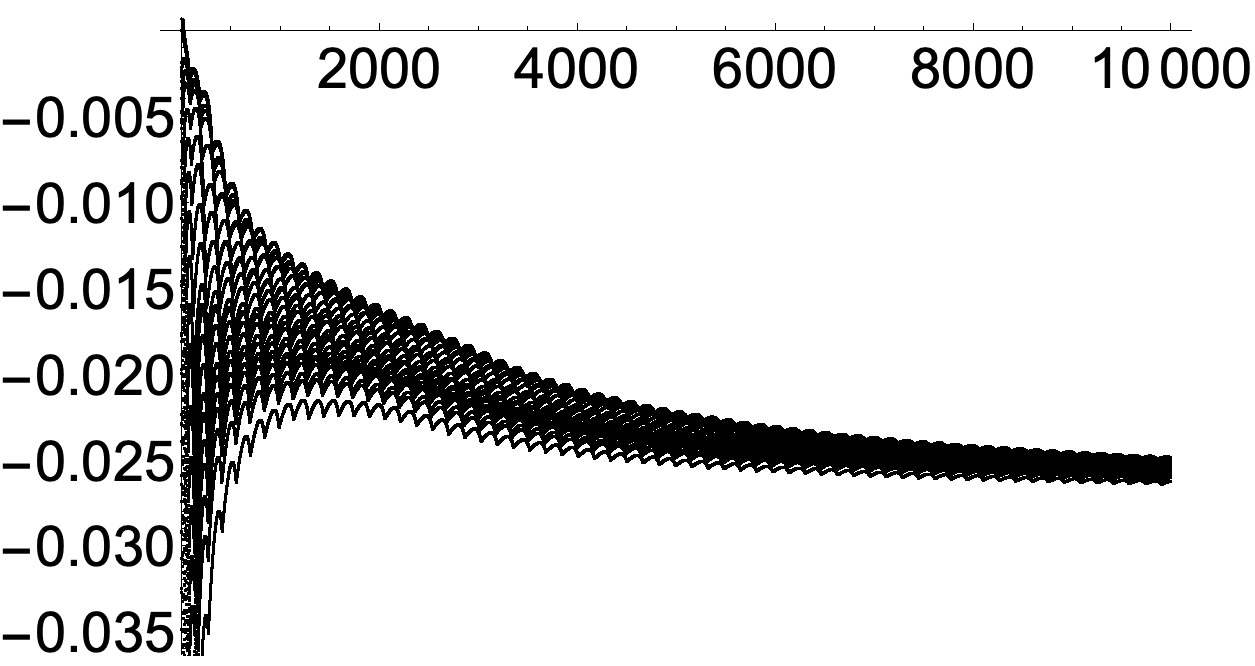}};
    \node at (0, -1.4) {$r = 0.2$};
    \node at (4, -1.4) {$r = 0.05$};
    \node at (8, -1.4) {$r = 0.025$};
       \end{tikzpicture}
        \caption{The convergence rate is stable for different values of $r$.}
        \label{fig:3}
    \end{figure}
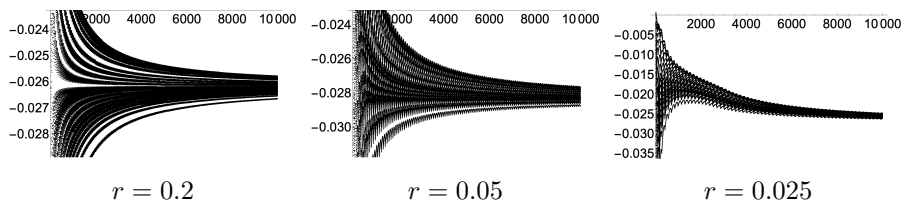
\end{center}

\subsection{Comparison to the standard gradient descent}
We quickly comment on how this accelerated method compares to standard gradient descent. It is clear that standard gradient descent with choice
$ \tau = 2/(\lambda_{\min} + \lambda_{\max})$
will always win any comparison; however, in practice we do not know $\lambda_{\min}$ and $\lambda_{\max}$ and we do not know how to choose $\tau$. We observe that once $\lambda_{\min} < \lambda_{\max}/2$, this modified method is always at least $50\%$ better in the allowable regime.

\begin{fact}
    If $\lambda_{\min} < \lambda_{\max}/2$ and $0 < \tau < 1/\lambda_{\max}$, then
    $$  c(\lambda_{\max}, \lambda_{\min}, \tau) \geq \frac{3}{2} \cdot \log \left| 1- \tau \lambda_{\min} \right| $$ 
\end{fact}

Since $ c(\lambda_{\max}, \lambda_{\min}, \tau)$ is defined implicitly via the solution of a fixed point equation involving special functions, it is nontrivial to work with (and we use `Fact' to indicate that it was obtained beyond any numerical doubt). More importantly, if either $\tau$ is very small, $\tau \ll 1/\lambda_{\max}$, or if the smallest eigenvalue is quite small, $\lambda_{\min} \ll \lambda_{\max}$, then the improvement grows and is unbounded. Intuitively, this is not surprising: the slow exponential growth will gradually find larger values that perform better; it is this regime (unknown $\lambda_{\min}, \lambda_{\max}$ and a value of $\tau$ that is perhaps a little bit too small) where we believe this modification of gradient descent to be particularly useful.
Classic gradient descent is very sensitive with regards to the choice of the stepsize $\tau$.  As is seen via  Taylor expansion
$$- \log \left( \frac{1}{1 - \tau \lambda_{\min}/\lambda_{\max}}\right) = - \frac{\lambda_{\min}}{\lambda_{\max}}\tau + \mathcal{O}(\tau^2);$$
while choosing $\tau$ to be too large will immediately lead to problems with convergence, choosing $\tau$ just a bit too small can dramatically slow down the convergence rate. The exponential stepsize turns out to be a very useful remedy: the dependence on the stepsize improves from linear to merely logarithmic.

\begin{corollary}[Stability for small stepsizes] If $\lambda_{\min}  \leq \lambda_{\max}/100$, then
$$\forall~0 < \tau < 1 \qquad \qquad c(\lambda_{\max}, \lambda_{\min}, \tau) \geq \frac{12 }{\log(13/\tau)} \frac{\lambda_{\min}}{\lambda_{\max}}.$$
\end{corollary}

The condition $\lambda_{\min}  \leq \lambda_{\max}/100$ was imposed to allow for a slightly more transparent proof; the proof could, with little work, be extended to slightly larger values of $\lambda_{\min}$. We illustrate this again by analyzing how the limiting convergence rate for the Rosenbrock function depends on $\tau$, see Table 1.

\begin{table}[h!]
\begin{tabular}{c|c|c|c|c}
$\tau$                      & 0.001    & 0.0001    & 0.00001   & 0.000001   \\
\hline
$c$ in $\exp(- c n)$         & 0.027 & 0.0191 & 0.01418 & 0.011249 \\    
\end{tabular}
\vspace{5pt}
\caption{The convergence rate of $\|x_n - x_{\min}\|$ for the Rosenbrock function as a function of the stepsize $\tau$.}
\end{table}

This may turn out to be tremendously useful in practice; finding the right stepsize $\tau$ is a consistent challenge when using gradient descent. If $\tau$ is too large, one does not converge, if $\tau$ is too small, then it takes forever. 
While stepsize remains a factor here, larger is still better, the logarithmic decay means that guessing too small has less of an impact. In the case of the Rosenbrock example above, we obtain a convergence rate better than standard gradient descent with $\tau = 0.001$ for the entire parameter range $10^{-26} = 0.00000000000000000000000001 \leq \tau \leq 0.001$.

\subsection{Related results.} This method does not seem to exist in the literature. We became interested in the problem because of recent work in the area of non-constant and `long' gradient steps \cite{Altschuler1, Altschuler3, bb, dasGupta, Grimmer, Grimmer2, mal, renegar}. These results are quite different, stepsizes are chosen a priori for all time. The closest (in spirit) existing idea is perhaps that of the \textit{bold driver} proposed by Vogl et al. \cite{vogl} and Battiti \cite{batt} in the late 1980s. The rule, with the typical parameters, is as follows
\begin{enumerate}
    \item Set $x_{n+1} = x_n - \tau_n \nabla f(x_n)$.
    \item If $f(x_{n+1}) < f(x_n)$, then $\tau_{n+1} = \frac{11}{10} \cdot \tau_n$, else $\tau_{n+1} = \frac{1}{2} \cdot \tau_n$.
\end{enumerate}
It is an interesting \textit{``quick and dirty'' method} (Battiti \cite{batt}) and performs quite well in practice. Its main difficulties are (a) that it is not clear how to choose the parameters ($1.1$ and $0.5$) and (b) a complete lack of theory. The method is fairly different from ours insofar as it requires a function evaluation. The bold driver heuristic has persisted in the literature, it has since been observed several time in the literature that in \textit{certain} applications an exponentially increasing stepsize can work \cite{LiExp2}. For functions that decay faster than quadratically, i.e. $f(x) = \|x\|^p$ with $p > 2$, this was proposed by Ho, Ren, Sanghavi, Sarkar and Ward \cite{HoExp}; they do not discuss any type of stopping/restarting criterion which means that their approach is not applicable to strictly convex functions, however, they note the phenomenon that we discuss in Figure \ref{fig:1} in their Figure 3. Restarts were considered, among others, by Donoghue-Candes \cite{do}, Renegar-Grimmer \cite{renegar} and  Roulet-d'Aspresmont \cite{roulet}. We also refer to the excellent survey by d'Aspremont-Scieur-Taylor \cite{Aspremont}.

\section{The Product Lemma}
We start with a precise analysis of the one-dimensional case: assume without loss of generality that $f(x) = \alpha x^2/2$ for some $\alpha > 0$ and that $x_0 > 0$.  The recursion formula is
\begin{align*}
   x_{n+1} = x_n - \tau e^{r n} f'(x_n) =  \left(1 -  \tau \alpha e^{r n} \right) x_n.
\end{align*}
Therefore, without restarts, we have
$$ x_{n+1} = x_0 \prod_{k=0}^{n}  \left(1 -  \tau \alpha e^{rk} \right).$$
The rest of this section is dedicated to analyzing the product
$$ \prod_{k=0}^{n}  \left(1 -  q e^{rk} \right) \qquad \mbox{where}~q >0~\mbox{and}~0 < r \ll 1.$$
We start with some general observations. First, in order to avoid over-oscillation of the gradient descent, we need $\tau \alpha < 1$ which suggests that the product to be analyzed is relevant for $0 < q < 1$. 
The parameter $r>0$ is user-specified; we may assume that it is small (say $r=0.01$ or even smaller than that). The subsequent analysis will show that the method does not depend very much on the precise value of $r$.
 Since $q = \tau \alpha$ and $\alpha$ is, in practice, not accessible, $q$ will be typically unknown. This means that, when choosing $r$, we cannot guarantee that the product will not be 0 for all $n$ sufficiently large; this is exceedingly unlikely (see below), however, it explains the absence of uniform estimates.

\begin{lemma}[Main Lemma] Let $0 < q < 1$ and $0 < r \ll 1$. Then, except for a small exceptional set of values of $r$, we have, as $r \rightarrow 0^+$, that
$$ X_n = \log \left|  \prod_{k=0}^{n}  \left(1 -  q e^{rk} \right) \right| $$
is approximated by
$$ X_n = \frac{1+o(1)}{r}\left( \Phi(q) - \Phi(q e^{rn})\right), $$
where $\Phi(y)$ denotes Spence's function
$$ \Phi(y) = - \int_0^y \frac{\log |1-z|}{z} dz.$$

\end{lemma}
  The approximation can only be true up to an exceptional set: if one of the terms happens to be unusually close to 0, then the product can become arbitrarily small; this is not a problem for the method (one converges unexpectedly faster). However, it is also extremely rare for this to happen as one would need to `accidentally' choose the correct parameters from an exceptionally small set (see \S 5.2). One may think of the Main Lemma as an analysis of the `worst case' which also happens to be the generic case.  The $1+o(1)$ convergence rate could be made more precise (involving a more detailed application of the Euler-Maclaurin formula), however, the convergence is so fast (see also Fig. \ref{fig:app}) that there was no need for that. Spence's function is classical, first published by the Scottish mathematician William Spence in 1809 \cite{spence}; it also arises organically in physics (see, for example, \cite{hooft}).

\begin{center}
\begin{figure}[h!]
    \begin{tikzpicture}
        \node at (0,0) {\includegraphics[width=0.5\textwidth]{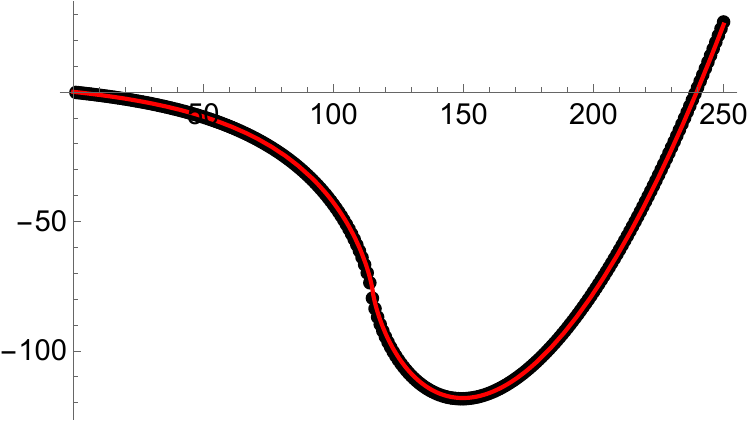}};
        \node at (4.8,-0.2) { $\log \left|\prod_{k=0}^{n} \left(1 - \frac{e^{0.02 k}}{10}\right) \right| $};
        \node at (5, -1) { \color{red} $\frac{1}{0.02} (\Phi(0.1) - \Phi(0.1 e^{0.02 n})$};
    \end{tikzpicture}
    \vspace{-10pt}
    \caption{The values of the product from Fig. 1 ($r=0.02, \tau =0.1$) in black, the approximation suggested by the Lemma (in red).}
    \label{fig:app}
\end{figure}    
\end{center}

\begin{proof}[Proof of the Main Lemma] We write
   \begin{align*}
 \sum_{k=0}^{n} \log \left| 1 - q  e^{k r} \right| \sim \int_0^n \log \left| 1 - q e^{kr} \right| dk
   \end{align*}
whose validity as an approximation (up to an exceptional set of $r$) is extensively discussed in \S 5.1 and \S 5.2. Substituting $z =  q e^{kr}$, we have
$$ \frac{dz}{dk} = qr e^{kr} = r z$$
and thus
$$  \int_0^n \log \left| 1 - q e^{kr} \right| dk = \frac{1}{r} \int_q^{q e^{ rn}} \frac{\log{|1 - z|}}{z} dz.$$
Introducing Spence's function
$$ \Phi(x) = - \int_0^x \frac{\log |1-u|}{u} du,$$
we arrive at
$$  \int_0^n \log \left| 1 - q e^{kr} \right| dk = \frac{1}{r}\left( \Phi(q) - \Phi(q e^{rn})\right).$$
\end{proof}
We also recall that Spence's function can also be explicitely expressed as
$$ \Phi(x) = -\int_0^x \frac{\log|1-u|}{u} du = \begin{cases} \mbox{Li}_2(x) \qquad &\mbox{if}~x \leq 1\\
\frac{\pi^2}{3}  - \mbox{Li}_2\left( \frac{1}{x} \right) - \frac{1}{2}(\log x)^2  \qquad &\mbox{if}~x > 1 \end{cases},$$
where $\mbox{Li}_2: (-\infty, 1)$ is the dilogarithm defined by
$$ \mbox{Li}_2(x) = \sum_{k=1}^{\infty} \frac{x^k}{k^2}.$$
This series expansion can be useful when doing asymptotic analysis of some special cases: for example $\mbox{Li}_2(x) = x + \mathcal{O}(x^2)$ for $x$ small.

\section{Proof of the Theorem}
\begin{proof}
We assume that
$$ f(x) = \frac{1}{2}\left\langle x, Qx \right\rangle$$
is a quadratic form with $Q \in \mathbb{R}^{d \times d}$ being a symmetric, positive-definite matrix. After a change of coordinates, we may assume without loss of generality that $Q$ is diagonal and
$ Q = \mbox{diag}(\lambda_1, \dots, \lambda_d),$
where $\lambda_1 \geq \dots \geq \lambda_d >0$ are the eigenvalues of $Q$.
We also assume that $x_0 \neq 0 \in \mathbb{R}^d$. 
The gradient descent
$$ x_{n+1} = x_n - \tau e^{rn} \nabla f(x_n)$$
is then diagonalized by the problem. In particular, we may assume that none of the coordinates of $x_0$ are 0 (since that coordinate would then stay 0 for all time and we would have reduced the problem to  a similar problem in dimension $d-1$). This leads to a factorization of the recursion: writing $x_n = (x_{n,1}, \dots, x_{n, d})$, a single recursive step is of the form
$$ x_{n+1, i} = x_{n,i} - \tau e^{rn} \lambda_i x_{n,i} = \left(1 - \tau e^{rn} \lambda_i\right) x_{n,i}.$$
Note that, in particular, we have
$$ x_{n,i} = x_{0,i} \prod_{k=0}^{n-1} \left(1 - \tau  \lambda_i e^{rk} \right) \qquad \qquad \qquad (\diamond)$$
which acts as pure multiplication on each coordinate.
We will start by assuming $x_0 \in \mathbb{R}^d$ to be arbitrary and will analyze at what point the stopping criterion $\|x_{n+1} - x_n\| \geq e^r\|x_n - x_{n-1}\|$ comes into effect. Once this happens, we may think of $x_n$ as a new starting point $x_0$ and the procedure continues. Observe that
$$ \|x_{n+1} - x_n\| = \tau e^{rn} \left( \sum_{i=1}^{d} \lambda_i^2 x_{n,i}^2 \right)^{1/2}.$$
The stopping criterion $\|x_{n+1} - x_n\| \geq e^r\|x_n - x_{n-1}\|$ can thus be rewritten as
$$ \tau e^{rn} \left( \sum_{i=1}^{d} \lambda_i^2 x_{n,i}^2 \right)^{1/2} \geq \tau e^r e^{r(n-1)} \left( \sum_{i=1}^{d} \lambda_i^2 x_{n-1,i}^2 \right)^{1/2}$$
which is equivalent to
$$  \|Qx_n\| =  \left( \sum_{i=1}^{d} \lambda_i^2 x_{n,i}^2 \right)^{1/2} \geq \left( \sum_{i=1}^{d} \lambda_i^2 x_{n-1,i}^2 \right)^{1/2} = \|Q x_{n-1}\|.$$
The goal is now, given $x_0 = (x_{0,1}, x_{0,2}, \dots, x_{0,d})$ to predict for what $n$ the stopping criterion becomes activated and what $x_n$ looks like at that point.  The equation $(\diamond)$ shows that each individual coordinate evolves in a deterministic fashion depending only on $\lambda_i$ and $r$. The absolute value of the product is initially decreasing and becomes quite small (as a function of $r$) until
$$ 1 - \tau \lambda_i e^{rk} \sim -1 \qquad \mbox{or} \qquad k \sim \frac{1}{r} \log\left( \frac{2}{\tau \lambda_i} \right)$$
after which it increases again (and exponentially quickly) and becomes unbounded. It is clear that this will happen the fastest for the largest eigenvalue and it will take more time for the smaller eigenvalues. In particular, the index $n$ at which point the cut-off criterion goes into effect has to satisfy
$$  \frac{1}{r} \log\left( \frac{2}{\tau \lambda_1} \right) \leq n \leq  \frac{1}{r} \log\left( \frac{2}{\tau \lambda_d} \right).$$
The first inequality is required since before that time all the coordinates are still decreasing (in absolute value); the second inequality follows in a symmetric fashion: after that, all the coordinates are increasing (in absolute value) and the cut-off criterion is guaranteed to come into effect.
We will now show that, given $x_0 \in \mathbb{R}^d$, for all $r$ sufficiently small (outside of the exceptional set), the heuristic sketched above is indeed correct in the largest and the smallest eigenvalue determine (up to lower order terms), the stopping criterion; moreover, $x_n$, after having been stopped, is mostly a linear combination of the largest and smallest eigenvector (up to small errors) which then gives precise control when restarting the method with $x_0 \leftarrow x_n$. This leads us to consider the function
$$ g(q) = \int_0^n \log \left| 1 - q e^{rx} \right| dx.$$
Clearly $g(0) = 0$ and it is also clear that $g'(0) < 0$ as well as the fact that $g'(q) > 0$ for all $q$ sufficiently large. The same integration by parts as in the Lemma shows  
$$  \int_0^n \log \left| 1 - q e^{kr} \right| dk = \frac{1}{r}\left( \Phi(q) - \Phi(q e^{rn})\right)$$
which can be reinterpreted as a rescaling of Spence's function: the Spence function is monotonically increasing, has a global maximum and is then monotonically decreasing. This shows that, due to the sign, $g(q)$ is first monotonically decreasing, has a global minimum and is then monotonically increasing.  Thus, for any interval $[a,b] \subset (0, \infty)$, we have
$$ \forall~a \leq q \leq b \qquad \qquad g(q) \leq \max\left\{g(a), g(b) \right\}.$$
This implies that for $r$ sufficiently small (up to a small exceptional set),
$$  \prod_{k=0}^{n} \left| 1 - \tau  \lambda_i e^{rk} \right| \leq \max\left\{\prod_{k=0}^{n} \left| 1 - \tau  \lambda_{\max} e^{rk} \right|, \prod_{k=0}^{n} \left| 1 - \tau  \lambda_{\min} e^{rk} \right| \right\}.$$
Moreover, when $\lambda_{\min} < \lambda_i < \lambda_{\max}$, then the difference gets exponentially more pronounced as $r \rightarrow 0^+$.  The simple bounds above imply that $n \sim 1/r$ which means that all these products scale like $\exp(- \mbox{const}/r)$ with different constants depending on the eigenvalue. This leads to the following behavior as $r \rightarrow 0^+$: initially, for small values of $n$, all the products $\prod_{k=0}^{n} \left| 1 - \tau  \lambda_i e^{rk} \right|$ are decaying at different rates; the product corresponding to $\lambda_{\max}$ is decaying the fastest, the product corresponding to $\lambda_{\min}$ is decaying the slowest. Then, once $k$ is so large that $\tau \lambda_{\max} e^{rk} \geq 2$, the product corresponding to the largest eigenvalue starts to increase again exponentially (the other products are still monotonically decreasing). Assuming $\lambda_{\max} = \lambda_1$ and $\lambda_{\min} = \lambda_d$, we have
$$  \|Qx_n\|^2 = \sum_{i=1}^{d} \lambda_i^2 x_{n,i}^2 = \lambda_1^2 x_{n,1}^2 + \lambda_d^2 x_{n,d}^2 + \mbox{exponentially smaller terms},$$
where exponentially smaller refers to exponentially with respect to the larger of the first two terms. This means that the stopping criterion becomes activated when the growth of $\lambda_1^2 x_{n,1}^2$ is faster than the decay of $\lambda_d^2 x_{n,d}^2$ (which is decaying the slowest and is therefore exponentially larger than all the other terms). This happens once
$$\prod_{k=0}^{n} \left| 1 - \tau  \lambda_{\max} e^{rk} \right| \qquad \mbox{and} \qquad \prod_{k=0}^{n} \left| 1 - \tau  \lambda_{\min} e^{rk} \right| \quad \mbox{become comparable in size,}$$
which depends on the values of $\lambda_{\min}, \lambda_{\max}, x_{0,1}, x_{0,d}$, the other coefficients and the other eigenvalues; however, these terms only enter polynomially and are thus, as $r  \rightarrow 0^+$, dominated by the terms that scale exponentially.
When this happens and the restart is triggered, all the other values are exponentially smaller and the procedure is restarted with a value $x_0$ which, up to exponentially smaller terms, has only two nonzero entries (one corresponding to the largest and one corresponding to the smallest eigenvalue).
It remains to understand the size of $n$ where the two products are comparable which, using the Main Lemma, leads to the equation (substituting $x = rn$)
$$ \Phi(\tau \lambda_{\min}) - \Phi(\tau \lambda_{\min} e^{x}) = \Phi(\tau \lambda_{\max}) - \Phi(\tau \lambda_{\max} e^{x}).$$ 
The exponential decay rate that both these quantity undergo is then exactly given by the product, i.e. it is, again as $r \rightarrow 0^+$,
$$ \prod_{k=0}^{n} \left| 1 - \tau  \lambda_{\min} e^{rk} \right| \sim \exp\left[ \frac{1}{r} \left( \Phi(\tau \lambda_{\min}) - \Phi(\tau \lambda_{\min} e^{x}) \right) \right].$$
However, we are interested in the asymptotic behavior as a function of the number of the number of steps $n$ and, here, the number of steps is itself unknown. This suggests considering the geometric mean of the decay which then, using $n = x/r$, leads to 
\begin{align*}
 \left( \prod_{k=0}^{n} \left| 1 - \tau  \lambda_{\min} e^{rk} \right|  \right)^{1/n} &\sim 
  \exp\left[  - \frac{1}{n} \frac{1}{r}  \left( \Phi(\tau \lambda_{\min} e^{x}) \right) - \Phi(\tau \lambda_{\min})\right] \\
  &=   \exp\left[  - \frac{1}{x}  \left( \Phi(\tau \lambda_{\min} e^{x}) \right) - \Phi(\tau \lambda_{\min})\right].
\end{align*}
This completes the argument.
\end{proof}

It is worth pointing out that the proof has a surprising similarity to the behavior of standard gradient descent in its simplest form. Recall that, in the notation of the proof of the Theorem, we are dealing with $ x_{n,i} = x_{0,i} \left(1- \tau \lambda_i\right)^n.$
If $\lambda_1 \geq \lambda_i \geq \lambda_d$, then the optimal stepsize $\tau$ is exactly the one where increasing or decreasing it leads to a larger exponential growth, i.e.
$$ 1 - \tau \lambda_{\min} = \left| 1 - \tau \lambda_{\max} \right| = \lambda_{\max} \tau - 1 \qquad \mbox{and thus} \qquad \tau = \frac{2}{\lambda_{\min} + \lambda_{\max}}.$$
It is worth pointing out that the proof is carried out for $r \rightarrow 0^+$ (outside the exceptional set). However, the argument is actually remarkably robust. We consider this with the quadratic form
$$ f(x,y,z) =  \frac{x^2}{2} + 2 \frac{y^2}{2} + 3 \frac{z^2}{2},$$
initial value $x_0 = (1,20,3)$, stepsize $\tau =0.1$ and $r=0.01$. The first restart is triggered after 245 steps. For comparison, the eigenvalues are $1,2,3$ and
the relevant equation is thus 
$$ \Phi(1/10) - \Phi( 1/10 \cdot e^{0.01 n}) = \Phi(3/10) - \Phi(3/10 \cdot e^{0.01 n})$$ 
which has the solution $ n \sim 241.822$. Just before the restart, we have
$a_{243} \sim (-5 \cdot 10^{-83}, -2 \cdot 10^{-97}, -7 \cdot 10^{-83})$
and $a_{244} \sim (6 \cdot 10^{-84}, 2 \cdot 10^{-97}, 1.7 \cdot 10^{-82})$
which shows that the growth induced by the largest eigenvalue outpaces the decay induced by the smallest (and the third eigenvalue is somewhere in between, leads to a much smaller contribution and does not play any rule).

\subsection{Proof of the Corollary.}
\begin{proof} 
We may assume without loss of generality that $\lambda_{\max} = 1$.  We additionally assume that $\lambda_{\min} < 0.1$. We are interested in solving the equation
$$ \Phi(\tau) - \Phi(\tau e^{x}) = \Phi(\tau \lambda_{\min}) - \Phi(\tau \lambda_{\min} e^{x}).$$
Using the alternative description
$$ \Phi(x) = -\int_0^x \frac{\log|1-u|}{u} du = \begin{cases} \mbox{Li}_2(x) \qquad &\mbox{if}~x \leq 1\\
\frac{\pi^2}{3}  - \mbox{Li}_2\left( \frac{1}{x} \right) - \frac{1}{2}(\log x)^2  \qquad &\mbox{if}~x > 1 \end{cases},$$
the relevant equation to solve is then
$$  \mbox{Li}_2\left( \frac{1}{\tau e^{x}} \right) + \frac{1}{2} \log{(\tau e^{x} )}^2 - \frac{\pi^2}{3} = \mbox{Li}_2(\lambda_{\min} \tau) - \mbox{Li}_2( \tau  e^{x }  \lambda_{\min} ).$$
We first note that we require $\tau e^x \geq 1$ for the left-hand side to be well-defined. Since $\tau \leq \lambda_{\max}$, this also implies $e^x \geq 1$ and $\lambda_{\min} \tau \leq \tau e^x \lambda_{\min}$. Introducing the function $h:[1, \infty] \rightarrow \mathbb{R}$ given by
$$ h(y) =  \mbox{Li}_2\left( 1/y\right) + \frac{1}{2} (\log{y })^2,$$
we note that $h(y) \leq 2.5$ for $y \leq 8$.  Therefore if $\tau e^x \leq 8$, then
$$ \mbox{Li}_2( \tau  e^{x }  \lambda_{\min} ) \leq \mbox{Li}_2( 8  \cdot \lambda_{\min} ) \leq \mbox{Li}_2( 1/2 )$$
and thus
\begin{align*}
      \mbox{Li}_2\left( \frac{1}{\tau e^{x}} \right) + \frac{1}{2} \log{(\tau e^{x} )}^2 +  \mbox{Li}_2( \tau  e^{x }  \lambda_{\min} ) &\leq 2.5 + \mbox{Li}_2(1/2)  < 3.27 = \frac{\pi^2}{3}
\end{align*}
which means that the equation does not have a solution.  We may thus assume that $\tau e^x \geq 8$. 
Since $\mbox{Li}_2$ is non-negative, any solution of the equation also satisfies
$$  \frac{1}{2} \log{(\tau e^{x} )}^2 + \mbox{Li}_2( \tau  e^{x }  \lambda_{\min} ) <  \frac{\pi^2}{3}$$
which, together with $\tau e^x \geq 8$, implies $\mbox{Li}_2( \tau  e^{x }  \lambda_{\min} ) < 1.13$ 
and $\tau e^x \lambda_{\min} < 0.85$. For any $0 < z < 0.85$, we have from the convexity of the dilogarithm that
$$ \mbox{Li}_2(z) \leq 1.4 z.$$
Therefore
$$ \left| \mbox{Li}_2(\lambda_{\min} \tau) - \mbox{Li}_2( \tau  e^{x }  \lambda_{\min} ) \right| \leq  \lambda_{\min}  1.5 (e^x + 1)\tau \leq 3\lambda_{\min} \cdot e^x \tau.$$
This allows us to treat, for $\lambda_{\min}$ small, the right-hand side as a perturbation of 0 and allows us to consider the simplified problem (see also Fig. \ref{fig:help})
$$  \mbox{Li}_2\left( 1/y\right) + \frac{1}{2} (\log{y })^2 = 0.$$
\begin{center}
    \begin{figure}[h!]
    \begin{tikzpicture}
    \node at (0,0) {\includegraphics[width=0.5\textwidth]{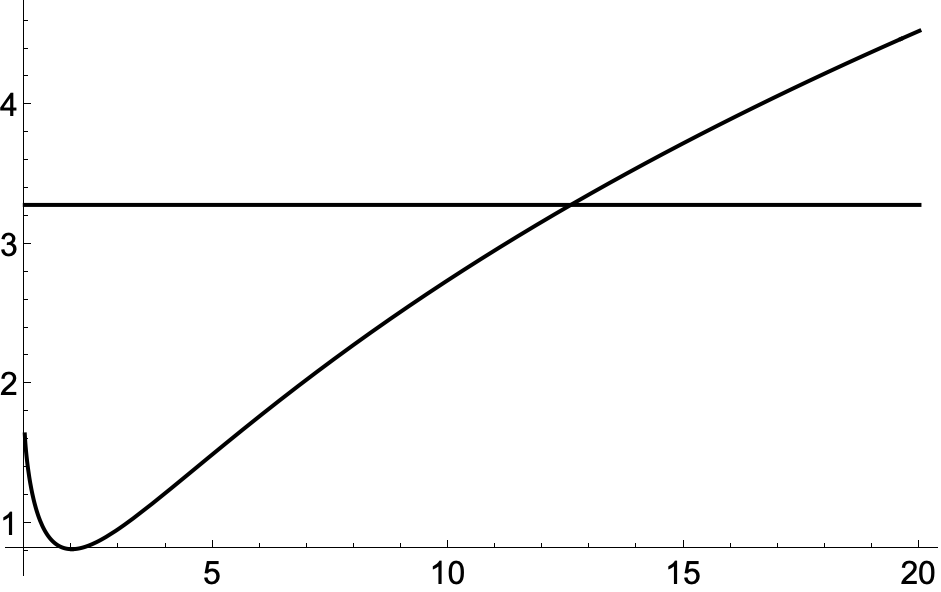}};
    \node at (5.1, 1.8) {$\mbox{Li}_2\left( 1/y\right) + \frac{1}{2} (\log{y })^2$};
        \node at (3.8, 0.7) {$\pi^2/3$};
    \end{tikzpicture}   
    \vspace{-10pt}
        \caption{Local stability analysis.}
        \label{fig:help}
    \end{figure}
\end{center}
This equation has the solution $y^* = 12.5952$ which shows together with the fact that $h'(y^*) \neq 0$
that the solution of the perturbed equation is given by $y^* \pm \mathcal{O}(\lambda_{\min})$. The main result implies
$$ c(\lambda_{\max}, \lambda_{\min}, \tau) = \frac{1}{x} \left(  \mbox{Li}_2(\lambda_{\min} \tau) - \mbox{Li}_2(e^{x } \lambda_{\min} \tau) \right).$$
which together with $e^x \tau \sim 12.5952 \pm \mathcal{O}(\lambda_{\min})$ implies that for $\lambda_{\min}$ sufficiently small
$$ c(\lambda_{\max}, \lambda_{\min}, \tau) \geq \frac{12 }{\log(13/\tau)} \frac{\lambda_{\min}}{\lambda_{\max}}.$$
\end{proof}

\section{Two examples}

\subsection{Regularized softplus.}  Consider the function
 $$ f(x,y) = \log(1 + e^x) + \log(1+e^y) + \frac{x^2+y^2}{1000}.$$
Its global minimum is in $x^* \sim (-4.66, -4.66)$. We consider standard gradient descent with $\tau = 1$ with 100 steps and compare it to the new version with $r = 0.01, 0.05, 0.1, 0.2$. We start $x_0 \in \mathbb{R}^2$ in 50 equispaced points on the disk with radius 10 centered at the origin. Standard gradient struggles to converge; the exponential version gets within distance $10^{-10}$ (independently of the value of $r$). With hindsight, one could increase the stepsize $\tau$ for the gradient flow; picking $\tau=20$ allows standard gradient descent to reach the minimum with an error of $\sim 10^{-8}$ after 100 steps. However, using $\tau = 20$ with exponential stepsizes leads to an error of $10^{-8}$ after at most 70 steps ($r=0.01$), 30 steps ($r=0.05$), 38 steps ($r=0.1$) and 35 steps ($r=0.2$) for all 50 initial points. We see that there is a difference in behavior when $r$ is small with respect to the number of stepsizes $e^{rn} = e^{0.7} \sim 2$ but that the behavior is otherwise fairly consistent across different values of $r$. We obtain an improvement in all cases.

\begin{center}
\begin{figure}[h!]
    \begin{tikzpicture}
\node at (0,0) {\includegraphics[width=0.3\textwidth]{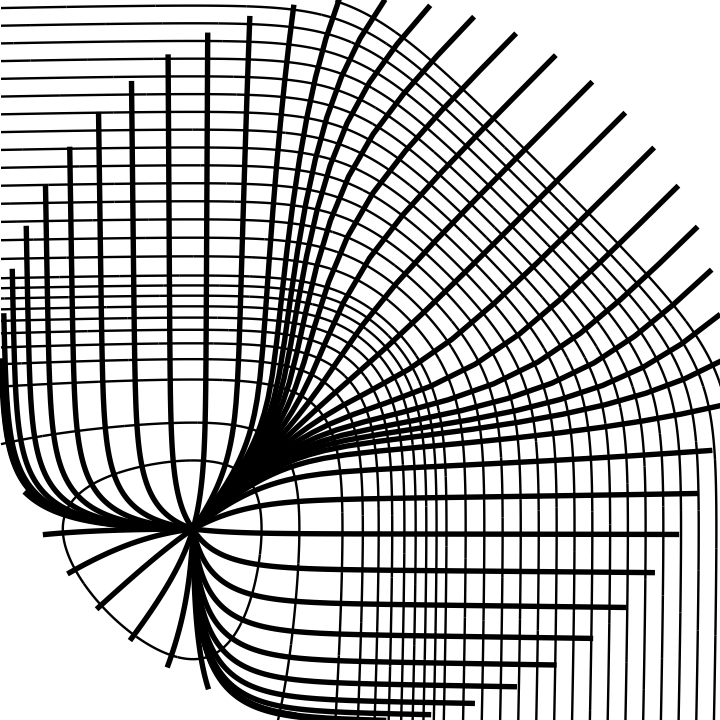}};
\node at (6,0) {\includegraphics[width=0.3\textwidth]{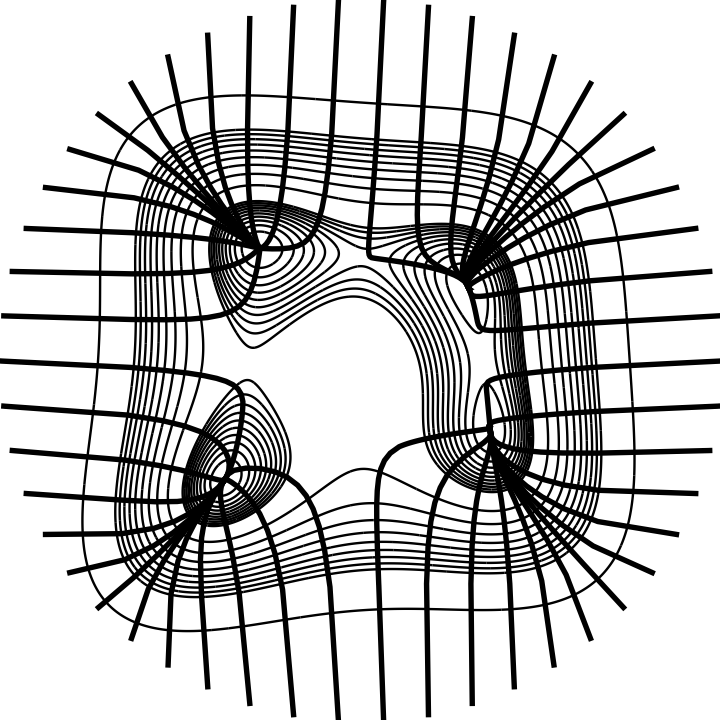}};
    \end{tikzpicture}
\caption{50 gradient flow vs level sets with respect to regularized softplus (left) and  Himmelblau (right).}    
\end{figure}
\end{center}

\subsection{Himmelblau function} We continue with a non-convex example due to Himmelblau \cite{himmelblau}. 
The Himmelblau function is given by
$$ f(x,y) = (x^2 + y - 11)^2 + (x+y^2 - 7)^2.$$
It has four local minima, all distance $\sim 5$ from the origin, and in all of them $f = 0$. We start again in 50 equispaced points on the disk with radius 10 centered around the origin. 
Standard gradient descent leads to some diverging trajectories when $\tau = 0.01$, we therefore choose $\tau = 0.001$, run the procedure for 100 steps and investigate the size of $f(x_{100})$. For these 50 initial conditions, roughly half the trajectories end up somewhere with $f(x_{100}) \sim 10^{-5}$ and roughly half satisfy $f(x_{100}) \sim 10^{-2}$ (and two outliers with $f(x_{100}) \geq 1$). It is interesting to compare this to the exponential stepsizes: when $r \in \left\{0.05, 0.1, 0.2\right\}$, we observe convergence to near machine precision for the same initial conditions. When $r=0.01$, we always end up with a smaller function value (at roughly twice the convergence rate).

\section{Comments and Remarks}

\subsection{Stability with respect to $r$.} The stability of the algorithm with respect to the choice
of $r$ is arguably one of its best features since it implies that the new parameter does not require
any fine-tuning (indeed, $r=0.01$ or $r=0.05$ seems to be a reasonable for a vast number of problems). The reason for this type of tremendous stability lies in the stability when arguing that
\begin{align*}
  \log \prod_{k=0}^{n}  \left|1 -  q e^{rk}\right| = \sum_{k=0}^{n} \log \left|1 -  q e^{rk} \right| \sim \int_0^n \log|1 - q e^{rx}| dx.
\end{align*}
The function $\log| 1 - q e^{rx} |$ is tremendously benign: outside the singularity it is smooth and monotonic, the singularity is only logarithmic. The only real source of error in the approximation is that a single summand may be unexpectedly close to the singularity -- and since the singularity is logarithmic, one would have to be exponentially close to the singularity for it to have any real effect.

\subsection{Lack of uniform control in $r$.} The purpose of this section is to construct an example showing that the exceptional set of $r$ in the Theorem is required. We consider the function
$$ f(x,y) = \frac{x^2}{2} + \frac{y^2}{10}$$
and we consider the gradient descent with exponentially increasing stepsizes and $\tau = 0.1$ and $r=0.1$.
Applying the Main Theorem predicts a decay rate of approximately $\exp(-0.413294 n)$ and this is what is being observed in Fig. \ref{fig:exc} (left). This should be constrasted to the slightly nearby value $r=0.1001123878$. Simply applying the Theorem again predicts the same convergence rate (since $r$ does not even feature as an object in the main result). However, numerically, we observe a slightly faster rate, something like $\exp(-0.43 n)$ ( Fig. \ref{fig:exc}, right).

\begin{center}
\begin{figure}[h!]
    \begin{tikzpicture}
\node at (0,0) {\includegraphics[width=0.45\textwidth]{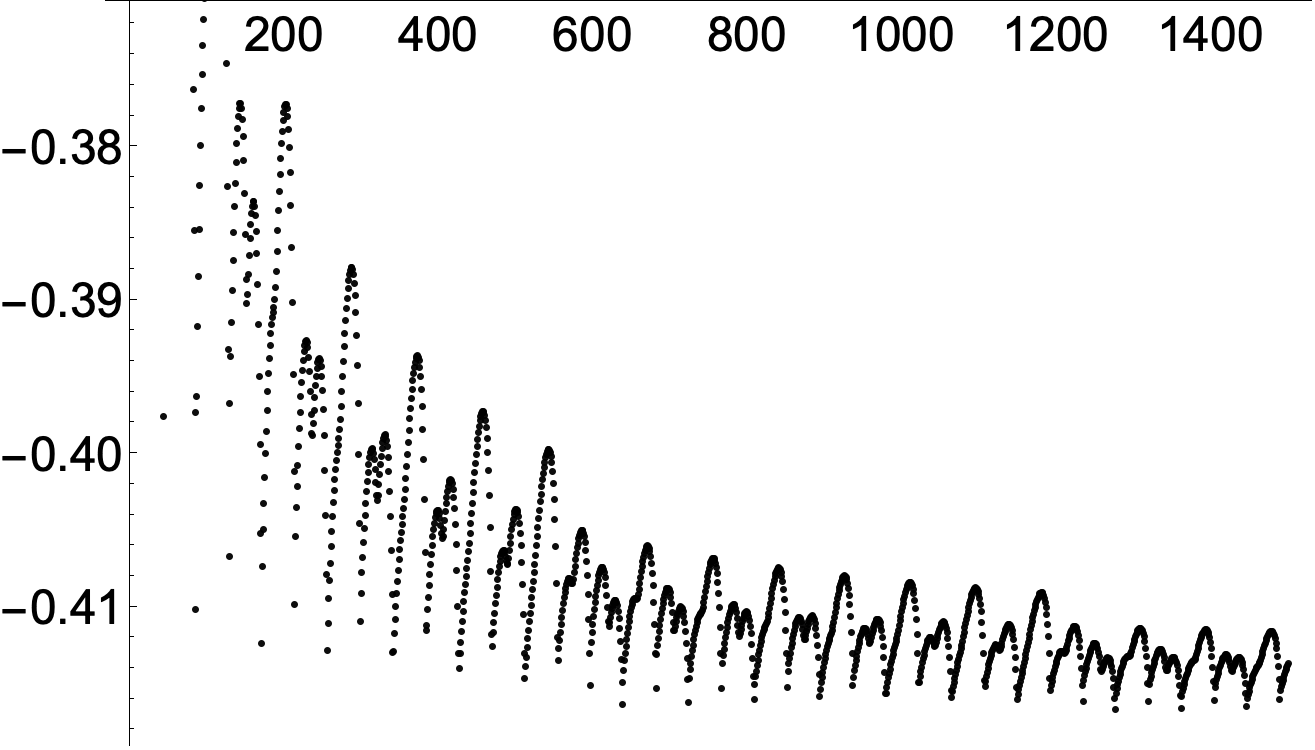}};
\node at (6,0) {\includegraphics[width=0.45\textwidth]{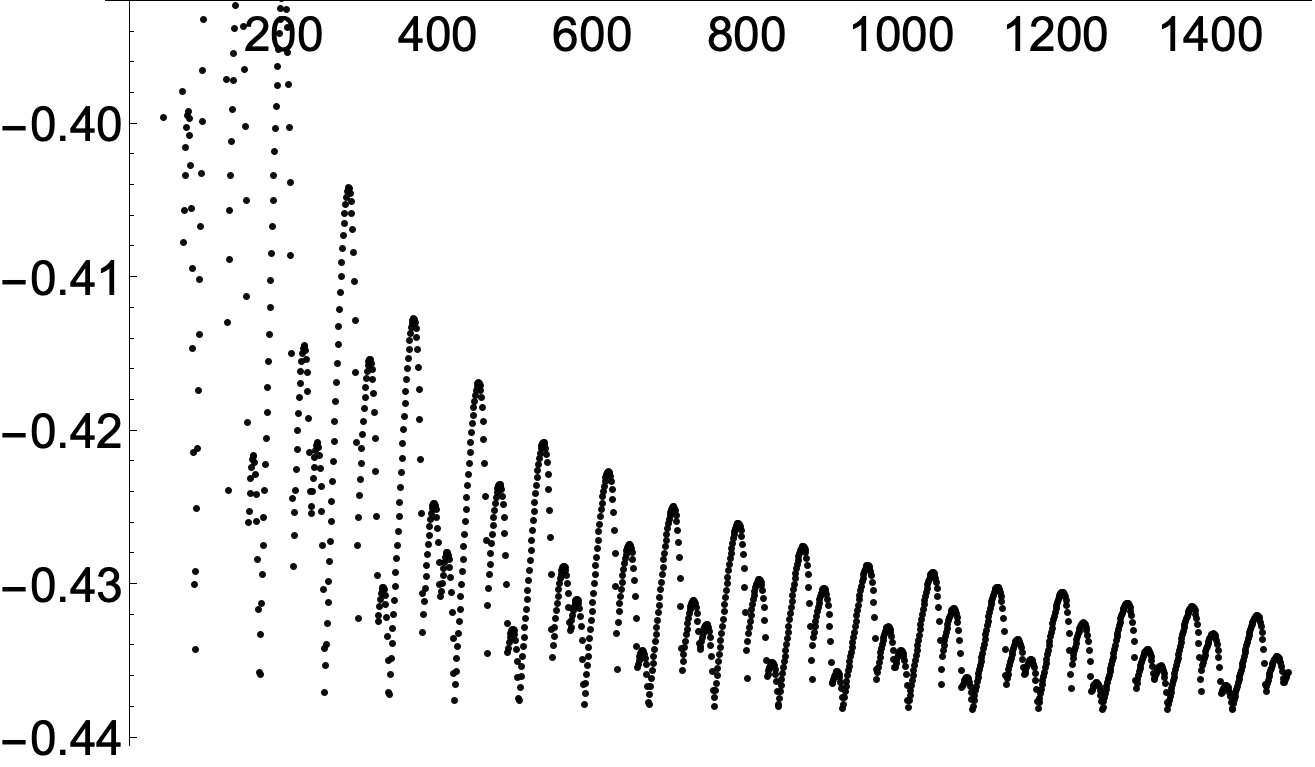}};
    \end{tikzpicture}
\caption{$n^{-1} \log\|x_n - x^*\|$ for $r=0.1$ (left) and the nearby exceptional value $r=0.1001123878$ (right). The example on the right has a slightly faster decay rate.}
\label{fig:exc}
\end{figure}
\end{center}

The reason for this discrepancy lies in the product 
$$ \prod_{k=0}^{n} \left(1 - \lambda_1 \tau e^{rk}\right) = \prod_{k=0}^{n} \left(1 - 0.1 e^{rk}\right) $$
and the 22nd term $(1-0.1 e^{22 r})$ which is small, of order $\sim 2 \cdot 10^{-3}$, but atypically small, of order $\sim 10^{-7}$, when $r=0.1001123878$ (and, indeed, this is how $r$ was constructed). This type of degeneracy requires many digits of accuracy in the computation to even be detected; while it is easy to see in the theory, it is difficult to find in practice; moreover, it seems like something that is difficult to exploit.

\subsection{Suboptimal behavior} \label{sec:count} We tried to find examples where the method fails compared to standard gradient descent. The method can perform suboptimally in one dimension: consider $f(x) = x^2/2$ and a stepsize $0 < \tau \ll 1$. Then, for $0 < r \ll 1$, the method will run $\log(2/\tau)/r$ steps before restarting: the typical decay will be governed by
$$ \frac{1}{2}\int_0^{2} \log| 1 - x| dx = - 1$$
predicting a typical decay of $\sim e^{-n}$. If $\tau = 1-\varepsilon$ is very cleverly chosen, and standard gradient descent is highly effective, the exponential stepsizes will still only result in a decay at rate $\sim e^{-n}$ (the computation is identical). The same reasoning applies to the \textit{bold driver} heuristic. However, the example is at least somewhat misleading because both methods draw their effectiveness from being able to play a slow direction against a fast direction, this requires at least $d \geq 2$ dimensions.


\begin{thebibliography}{10}

\bibitem{Altschuler1} J. M. Altschuler and P. A. Parrilo, Acceleration by stepsize hedging: Silver Stepsize Schedule for smooth convex optimization: JM Altschuler, PA Parrilo. Mathematical Programming, 213 (2025), p. 1105-1118.

\bibitem{Altschuler3} J. M. Altschuler and P. A. Parrilo. Acceleration by stepsize hedging: Multi-step descent and the silver stepsize schedule. Journal of the ACM, 72 (2025): p. 1-38.

\bibitem{Aspremont} A. d'Aspremont, D. Scieur and A. Taylor, Acceleration methods, Foundations and Trends in Optimization, 5(1-2), p1–245, 2021.


\bibitem{bb} J. Barzilai and J. M. Borwein, Two-point step size gradient methods. IMA journal of Numerical Analysis 8 (1988), 141-148.


\bibitem{batt} R. Battiti,. Accelerated backpropagation learning: Two optimization methods. Complex systems 3 (1989), 331-342.

 
\bibitem{dasGupta} S. Das Gupta, B. P. G. Van Parys, and E. K. Ryu, Branch-and-bound performance estimation programming: a unified methodology for constructing optimal optimization methods. Mathematical Programming, 204(1–2):567–639, 2023

\bibitem{do} B. O’Donoghue and E. Candes,  Adaptive restart for accelerated gradient schemes. Foundations of Computational Mathematics 15 (2015), 715-732.

\bibitem{gis} P. Giselsson and S. Boyd, Monotonicity and restart in fast gradient methods, In 53rd IEEE Conference on Decision and Control, pp. 5058-5063. IEEE, 2014.

    
\bibitem{Grimmer} B. Grimmer. Provably faster gradient descent via long steps. SIAM Journal on Optimization, 34(3):2588–2608, 2024.
    
\bibitem{Grimmer2} B. Grimmer, K. Shu, and A. L. Wang. Accelerated objective gap and gradient norm convergence for gradient descent via long steps. INFORMS Journal on Optimization, 7(2):156–169, 2025.

\bibitem{himmelblau} D. Himmelblau, Applied Nonlinear Programming, McGraw-Hill, 1972.

\bibitem{HoExp} N. Ho and T. Ren and S. Sanghavi and P. Sarkar and R. Ward, An Exponentially Increasing Step-size for Parameter Estimation in Statistical Models, 2022, \url{https://arxiv.org/pdf/2205.07999}

  \bibitem{hooft} G. 't Hooft and M. Veltman, Scalar one-loop integrals. Nuclear Physics B, 153 (1979), p. 365-401.
     

\bibitem{LiExp2} Z. Li, K. Lyu and S. Arora, Reconciling modern deep learning with traditional optimization analyses: The intrinsic learning rate. Advances in Neural Information Processing Systems 33 (2020), 14544-14555.


\bibitem{mal} Y. Malitsky and K. Mishchenko, Adaptive Gradient Descent without Descent, Proceedings of the 37 th International Conference on Machine Learning, PMLR 119, 2020. 

\bibitem{renegar} J. Renegar and B. Grimmer, A simple nearly optimal restart scheme for speeding up first-order methods. Foundations of computational mathematics, 22 (2022), p. 211-256.

\bibitem{rosenbrock} H. Rosenbrock, An automatic method for finding the greatest or least value of a function, The Computer Journal 3 (1960): p. 175–184.

\bibitem{roulet} V. Roulet and A. d'Aspremont, Sharpness, restart and acceleration, Advances in Neural Information Processing Systems 30 (2017).

\bibitem{spence} W. Spence, Spence, An Essay on the Theory of the Various Orders of Logarithmic Transcendents: With an Inquiry Into Their Applications to the Integral Calculus and the Summation of Series, John Murray and Archibald Constable and Company, 1809.


\bibitem{vogl} T. Vogl, J. K. Mangis, A. K. Rigler, W. T. Zink, and D. L. Alkon, Accelerating the convergence of the back-propagation method, Biological cybernetics 59 (1988), p. 257-263.

\end{thebibliography}
\end{document}